\documentclass[11pt, a4paper]{amsart}

\usepackage{amsmath}
\usepackage{latexsym}
\usepackage[all]{xy}
\usepackage{color}
\newtheorem{teorema}{Theorem}[section]

\include{amslatex}
\newtheorem{proposicion}[teorema]{Proposition}

\begin{document}
\title[On Clifford bundles and Finsler spaces]{A note on Clifford bundles and certain Finsler type spaces}

\maketitle
\begin{center}
\author{Ricardo Gallego Torrom\'e\footnote{email: rigato39@gmail.com}}
\end{center}\emph{}
\begin{center}
\address{Department of Mathematics\\
Faculty of Mathematics, Natural Sciences and Information Technologies\\
University of Primorska, Koper, Slovenia}
\end{center}

\begin{abstract}
We study a relation between certain extensions of the Clifford bundle and Finsler type structures that naturally generalize the standard Clifford relation between (pseudo)-Riemannian metric structures and Dirac matrices. We show for flat metrics that there is a triangle map between Finsler structures constructed from an (pseudo)-Riemannian metric and $1$-forms on $M$, the extension of the Clifford bundle and relevant Dirac type operators.
\end{abstract}

\section{Introduction}
 Let $(M,g)$ be a four dimensional (pseudo)-Riemannian space. The Clifford bundle is constructed canonically from the Clifford algebras $C(T_x M,g_x)$ associated to each fiber $T_xM$ or, due to the isomorphism induced by $g$, from the dual bundle with fibers $T^*M$ \cite{Benn Tucker 1988}. After introducing natural coordinates, a generic tangent vector $y=\,y^i\,\frac{\partial}{\partial x^i}|_{x}\in \,T_xM$ is represented by the components $\{y^\mu,\,\mu=1,2,3,4\}$. Then for each Clifford algebra  $C(T_x M,g_x)$ we have  pointwise relations
\begin{align}
y^\mu{\bf e}_\mu (x) \,  z^\nu{\bf e}_\nu(x)+\,z^\nu{\bf e}_\nu(x)\, y^\mu {\bf e}_\mu(x) =\,2\,g_x(y,z)
\label{Clifford relation}
\end{align}
for each $y,z\in\,T_x M$ and where ${\bf e}_\mu$ for $\mu =\,1,2,3,4$ are the generators of the Clifford algebra $C(T_x M,g_x)$. In this work we report a generalization of the Clifford relation \eqref{Clifford relation} to certain sheaves constructed from the geometric data of certain types of Finsler spacetimes in dimension four.

The theory of Finsler spaces of positive signature can be formulated in the context of robust geometric frameworks \cite{Berwald,Cartan Exposes II,Chern1948}. The case of Finsler structures with Lorentzian metrics still lacks of a general framework, but such models have found many applications in mathematical physics, ranging from application in emergent quantum mechanics \cite{Ricardo06,Ricardo2014}, geometries of extended relativistic gravitational models \cite{GallegoPiccioneVitorio:2012}, dual dispersion relations \cite{Pfeifer 2019} and as a spacetime geometry compatible with clock hypothesis of relativity \cite{Ricardo 2017}. It is this ample range of applicability of the Finsler concept of space that partially motivates many studies aiming to extend (pseudo)-Riemannian theory to the Finsler setting.

We show in this work that the Clifford relation \eqref{Clifford relation} can be extended to Randers spaces and other types of Finsler spaces determined by vector fields defined on $M$. This extension is achieved by considering certain sub-algebras of a free algebra constructed from the Clifford algebra $C(T_x M,g_x)$ for each $x\in\,M$ and the trace operator. We also show that the elements of the algebra associated to Finsler type metrics are related with partial first order differential operators. Such relation is established for flat metrics, but the theory can be systematically extended to arbitrary spin manifolds.

The construction of Clifford algebras and Clifford bundles from Finsler type structures and Lagrange structures has been investigated by several authors \cite{Vacaru Stavrinos Gaburov Gonta}. However, The method followed in this work makes use of the standard Clifford bundle of a (pseudo)-Riemannian metric and the trace operator to define a particular sheaf, instead of providing a direct generalization using the Lagrange structure.

 Finally, let us remark that a formal point of view has been adopted, discussing a generalization of the relation \eqref{Clifford relation} to certain spaces where formal Finsler structures are defined, but disregarding most of the regularity issues concerning the theory of Finsler structures with Lorentzian signature.

 \section{On certain algebras obtained from the Clifford algebra in dimension four} Let $(M,g)$ be a four dimensional Lorentzian space. Most of our discussion requires only local methods. Therefore, we will make extensive use of local pointwise, orthonormal frames, where $g_x :=g|_x$ for  $x\in\,M$ is diagonal. For the components of geometric objects  in local orthonormal frames we use Latin indices $\{y^i,\,i=1,2,3,4\}$.  From now on we are going to work on orthonomal frames, if anything else is not stated. We also adopt Einstein's index convention.
 In an orthonormal frame at $x\in M$ the Clifford relations adopt the form
 \begin{align}
 \gamma_i\, \gamma_j+\gamma_j\,\gamma_i=\,2\,\eta_{ij}\,\mathbb{I}_4,\quad i,j=1,2,3,4,
 \label{Clifford relation in orthogonal}
 \end{align}
 where $\eta_{ij}$ are the components of the metric $g$ in the orthonormal frame and $\gamma_i =\,1,2,3,4$ are $4\times 4$ Dirac gamma matrices.  In this setting, let us consider the following linear combinations:
    \begin{displaymath}
    \mathcal{M}=\frac{\gamma_i y^i}{|\eta(y,y)|^{1/2}}-\mathbb{I}_4,\quad\tilde{\mathcal{M}}=\frac{\gamma_i y^i}{|\eta(y,y)|^{1/2}}+\mathbb{I}_4,
    \end{displaymath}
    where $\eta(y,y):=\,\eta_{ij}\,y^i\,y^j$ for each $y\in\,T_xM$ such that $\eta(y,y)\neq 0$.
    One can also form the linear combinations
    \begin{displaymath}
    \mathcal{F}_A=\gamma_i y^i-A_iy^i\,\mathbb{I}_4,\quad \tilde{\mathcal{F}}_A=\gamma_i y^i\,+A_iy^i\,\mathbb{I}_4,
    \end{displaymath}
    \begin{displaymath}
    \mathcal{F}_{AB}=\gamma_iy^i\,\gamma_j y^j-(A_iy^i)(B_jy^j) \mathbb{I}_4,\quad \tilde{\mathcal{F}}_{AB}=\gamma_iy^i\,\gamma_j y^j+(A_iy^i)(B_jy^j) \mathbb{I}_4
    \end{displaymath}
    and similarly for $\mathcal{F}_{ABC}$, etc,... where $A,B,C,...\in \,\Lambda^1 M$. The combinations $\mathcal{M}$, $\tilde{\mathcal{M}}$, $\mathcal{F}_A$, $\tilde{\mathcal{F}}_A$, $\mathcal{F}_{AB}$, $\tilde{\mathcal{F}}_{AB}$, etc...  generate the complex free algebra
     \begin{align*}
     \mathcal{A}_\mathbb{C}(x,y) :=& \{\lambda_{\mathcal{M}} \mathcal{M}+\,{\lambda}_{\tilde{\mathcal{M}}}\tilde{\mathcal{M}}+\,\lambda_{\mathcal{F}_A} \mathcal{F}_A+\,{\lambda}_{\tilde{\mathcal{F}}_A}\tilde{\mathcal{F}}+\lambda_{\mathcal{M}\,\mathcal{F}} \mathcal{M}\cdot \mathcal{F}_A\,\\
& +\lambda_{\mathcal{M}\tilde{\mathcal{F}}_A}\,\mathcal{M}\cdot\tilde{\mathcal{F}}_A\,+\lambda_{\tilde{\mathcal{M}}\mathcal{F}_A}\,\tilde{\mathcal{M}}\cdot\mathcal{F}_A\,+\lambda_{\tilde{\mathcal{M}}\,\tilde{\mathcal{F}}} \tilde{\mathcal{M}}\cdot \tilde{\mathcal{F}}_A\\
& +\lambda_{\mathcal{F}_{AB}} \mathcal{F}_{AB}\,+...,\\
&  \lambda_{\mathcal{M}},\,\lambda_{\tilde{\mathcal{M}}},\,\lambda_{\mathcal{F}_A},\lambda_{\tilde{\mathcal{F}}_A},\,
\lambda_{\mathcal{M}\mathcal{F}_A},\,\lambda_{\mathcal{F}\tilde{\mathcal{F}}_A},\,\lambda_{\tilde{\mathcal{M}}\tilde{\mathcal{F}}_A},\,\lambda_{\tilde{\mathcal{M}}\mathcal{F}_A},\,\lambda_{\mathcal{F}_{AB}},...\in \mathbb{C}\}
      \end{align*}
      for each $y\in\,T_xM$ and $x\in\,M$.
Note that the algebra $\mathcal{A}_\mathbb{C} (x,y)$ is generated by $y$-homogeneous elements. Indeed, the algebra can be decomposed as
\begin{align*}
\mathcal{A}_\mathbb{C} (x,y)=\,\oplus^{+\infty}_{k=0}\,\mathcal{A}^k_\mathbb{C}(x,y),
\end{align*}
with
\begin{align*}
 \mathcal{A}^0_\mathbb{C}(x,y)=\,&\left\{\lambda_{\mathcal{M}}\,\mathcal{M}+\,\lambda_{\tilde{\mathcal{M}}}\,\tilde{\mathcal{M}},\,\lambda_{\mathcal{M}},\lambda_{\tilde{\mathcal{M}}}\in\,\mathbb{C}\right\},\\
& \mathcal{A}^1_\mathbb{C}(x,y)=\,\{\lambda_{\mathcal{M}\mathcal{F}_A}\,\mathcal{M}\cdot\mathcal{F}_A\,+
\lambda_{\mathcal{M}\tilde{\mathcal{F}}_A}\,\mathcal{M}\cdot\tilde{\mathcal{F}}_A\,+\lambda_{\tilde{\mathcal{M}}\mathcal{F}_A}\,\tilde{\mathcal{M}}\cdot\mathcal{F}_A\, \\
&+\lambda_{\tilde{\mathcal{M}}\tilde{\mathcal{F}}_A}\,\tilde{\mathcal{M}}\cdot\tilde{\mathcal{F}}_A\,+\lambda_{\mathcal{F}}\,\mathcal{F}_A\,+\lambda_{\tilde{\mathcal{F}}_A}\,\tilde{\mathcal{F}}_A,\\
&\lambda_{\mathcal{M}\mathcal{F}_A},\,\lambda_{\mathcal{F}\tilde{\mathcal{F}}_A},\,\lambda_{\tilde{\mathcal{M}}\tilde{\mathcal{F}}_A},\,\lambda_{\tilde{\mathcal{M}}\mathcal{F}_A},\,\lambda_{\mathcal{F}},\,\lambda_{\tilde{\mathcal{F}}_A}\,
\in\,\mathbb{C} \},\\
& ...
\end{align*}
Each $\mathcal{A}^k_\mathbb{C}(x,y)$ is an $y$-homogeneous of degree $k$, complex vector space.
Furthermore,  $\mathcal{A}_\mathbb{C}(x,y)$ is a commutative algebra,
\begin{proposicion}
$\mathcal{A}_\mathbb{C}(x,y)$ is a graded, commutative algebra:
\begin{align*}
\left[\,\sum^{+\infty}_{k=1}\,\alpha_k \mathcal{G}_k,\sum^{+\infty}_{l=1}\,\beta_l \mathcal{G}_l\right]=0,
\end{align*}
for any pair of finite linear combinations of the form
\begin{align*}
\sum^{+\infty}_{k=1}\,\alpha_k \mathcal{G}_k =\,&\alpha_{\mathcal{M}} \mathcal{M}+\,\tilde{\alpha}_{\tilde{\mathcal{M}}}\tilde{\mathcal{M}}+\,\alpha_A \mathcal{F}_A+\,\tilde{\alpha}_{\tilde{\mathcal{F}}}\tilde{\mathcal{F}}+\\
     &+\alpha_{\mathcal{M}\,\mathcal{F}} M\cdot \mathcal{F}_A\,+\alpha_{AB} \mathcal{F}_{AB}\,+...
\end{align*}
\begin{align*}
\sum^{+\infty}_{l=1}\,\beta_l \mathcal{G}_l =\,&\beta_{\mathcal{M}} \mathcal{M}+\,\tilde{\beta}_{\tilde{\mathcal{M}}}\tilde{\mathcal{M}}+\,\beta_A \mathcal{F}_A+\,\tilde{\beta}_{\tilde{\mathcal{F}}}\tilde{\mathcal{F}}+\\
     &+\beta_{\mathcal{M}\,\mathcal{F}} M\cdot \mathcal{F}_A\,+\beta_{AB} \mathcal{F}_{AB}\,+...
\end{align*}
with only a number of coefficients in $\{\alpha_k,\,k=1,2,...\}$ and $\{\beta_k=1,2,...\}$ non-zero.
\end{proposicion}
\begin{proof}
The following relations follow straightforward,
\begin{align*}
& [\mathcal{M},\tilde{\mathcal{M}}]=0,\,
 [\mathcal{M},\mathcal{F}_A]=0,\,
 [\tilde{\mathcal{M}},\mathcal{F}_A]=0,\\
& [F_A,\mathcal{F}_B]=0,\,
[\mathcal{M},\mathcal{F}_{AB}]=0,\,
 ....,
\end{align*}
for $A,B,...\in\,\Lambda^1 M$, from which the result follows.
\end{proof}
The  collection of algebras $\mathcal{A}_{\mathbb{C}}:=\{\mathcal{A}_\mathbb{C}(x,y),\,(x,y)\in\,TM\setminus NC\}$ determines a sheaf over bundle $TM\setminus NC\to M$,
 \begin{align*}
 \pi_\mathcal{A}:\mathcal{A}_{\mathbb{C}}\to \,TM\setminus NC.
 \end{align*}
 The  stalk over $(x,y)$ is the algebra $\mathcal{A}_\mathbb{C}(x,y)$. We call $\mathcal{A}_{\mathbb{C}}$ the {\it Clifford sheaf} over $TM\setminus NC$.

The trace operation
\begin{align*}
Tr:\mathcal{A}_{\mathbb{C}}\to\,\mathbb{C}
\end{align*}
satisfies the usual rules as for Dirac matrices. Specifically, the trace operator
is a cyclic operator such that in an orthonormal basis the generators $\{\gamma_i,\,i=1,2,3,4\}$ satisfies usual Dirac type relations,
 \begin{align*}
 & Tr (\gamma_i)=0,\quad i=1,2,3,4,\\
 & Tr (\gamma_i\,\gamma_j)=\,4 \,\eta_{ij},\quad i,j=1,2,3,4,\\
 & Tr \left(\gamma_i\,\gamma_j\,\gamma_k\right)=0,\quad i,j,k=1,2,3,4,\\
 & Tr \left(\gamma_i\,\gamma_j\,\gamma_k\,\gamma_l\right)=\,4\left(\eta_{ij}\,\eta_{kl}\,-\eta_{ik}\,\eta_{jl}+\,\eta_{il}\,\eta_{jk}\right),\quad i,j,k,l=\,1,2,3,4,\\
 & ...
 \end{align*}
 These relations will play a relevant role in our considerations.
\begin{proposicion}\label{Randers spinor relation}
For $|g(y,y)|^{1/2}\neq 0$, the following relation holds:
\begin{align}
Tr(\mathcal{M}\cdot\mathcal{F}_A)= \,4\left(\mp\,|g(y,y)|^{1/2}+\,A_j\,y^j\right),
\end{align}
where the negative sign corresponds to the case when $\eta(y,y)<0$ and the positive sign corresponds to the case $\eta(y,y)>0$.
     \end{proposicion}
\begin{proof}We consider orthonormal coordinates, where the metric $g$ has components $\eta_{ij}$.
By direct computation, using the standard properties of the trace of the Dirac matrices and since $g(x,y)$ is numerically equal to $\eta_{ij}y^i\,y^j$, we have
\begin{align*}
Tr\left(\mathcal{M}\,\mathcal{F}_A\right) & =\,Tr\left(\left(\frac{\gamma_i y^i}{|g(y,y)|^{1/2}}+\mathbb{I}_4\right)\cdot \left( \gamma_j y^j-\,A_j\, y^j\,\mathbb{I}_4\right) \right)\\
& =\,4\,\left(\frac{\eta_{ij}y^iy^j}{|g(y,y)|^{1/2}}+ A_j\,y^j\right)\\
& =\,4\,\left(\mp|g(y,y)|^{1/2}+\,A_j\,y^j\right),
\end{align*}
where the negative sign corresponds to $\eta_{ij}y^i\,y^j<0$ and the positive sign corresponds to $\eta_{ij}\,y^i y^j>0$. Then the result follows.
\end{proof}
Similarly, one obtains the relation
\begin{align}
Tr(\mathcal{M}\cdot\tilde{\mathcal{F}}_A)= \,4\left(\mp\,|g(y,y)|^{1/2}-\,A_j\,y^j\right),
\end{align}
for $g(y,y)<0$.
\section{Examples}
We discuss in this section examples of metric structures that can be described within the framework of the algebra $\mathcal{A}_{\mathbb{C}}$. In all these examples, it is assumed that $(M,g)$ is either a Riemannian or a Lorentzian space. Hence we use the Clifford relation in an orthonormal basis \eqref{Clifford relation in orthogonal} in all our calculations.
\subsection{Lorentzian manifold}
The element $\mathcal{M}\in\,\mathcal{A}_{\mathbb{C}}$ provides examples of Riemannian norms,
     \begin{align}
     Tr(\mathcal{M}\cdot\tilde{\mathcal{M}})+4=\,\pm4|g(y,y)|^{1/2},
     \label{MM*}
     \end{align}
     the sign depending on the temporal $(-)$ or spacelike $(+)$ character of $y\in \,T_xM$.
     Similarly, we have
         \begin{align}
     Tr(\mathcal{M}\cdot{\mathcal{M}})-4=\,\pm 4|g(y,y)|^{1/2},
     \label{MM}
     \end{align}
     For null vectors $g(y,y)$, one defines the limit
     \begin{align}
     \lim_{g(y,y)\to 0}\, \frac{1}{2}\,Tr(\mathcal{M}\cdot{\mathcal{M}})-1=\,\frac{1}{2}\,Tr(\mathcal{M}\cdot\tilde{\mathcal{M}})-1 =\,0.
     \end{align}
Therefore, the metric structure $(M,g)$ is encoded in $\mathcal{M}\in\,\mathcal{A}_{\mathbb{C}}$.

\subsection{Angular metric type structure}
Given a vector field $A$, the trace
 $Tr(\mathcal{F}_A\,\tilde{\mathcal{F}}_A)$ formally provides a physical relevant metric,
     \begin{align}
     Tr(\mathcal{F}_A\cdot \tilde{\mathcal{F}}_A)=\,4\left(g(y,y)-(A\cdot y)^2\right),
     \label{SpinorBeam fluid}
     \end{align}
 where for the $1$-form $A\in \,\Lambda M$ and each tangent vector $x\in\,T_xM$, one has $A\cdot y :=\,A_i (x)\,y^i$.
In order that $(M,g)$ is a non-degenerate spacetime structure, further constrains on $A$ must be imposed. In particular, it is required that the Hessian of $L_A$
\begin{align*}
L_A(x,y)=g(y,y)-(A\cdot y)^2
\end{align*}
to  be non-degenerated. Since the fundamental tensor of $L_A$ is of the form
\begin{align*}
g_{ij}:=\frac{1}{2}\,\frac{\partial^2 L_A(x,y)}{\partial y^i\partial y^j}=\,\eta_{ij}(x)-A_i(x)A_j(x),
\end{align*}
the sufficient condition for this regularity is that $|g(y,y)|<1$.

 Similarly,
      \begin{align}
     Tr(\mathcal{F}_A\cdot {\mathcal{F}}_A)=\,4\left(g(y,y)+(A\cdot y)^2\right).
      \label{SpinorBeam fluid 2}
     \end{align}
The Hessian of the corresponding Finsler structure is regular if $|g^*(A,A)|<1$, where $g^*$ is the dual metric of $g$.

\subsection{Randers type structures}
An Asanov-Randers space \cite{Ricardo Randers-Lorentz} is characterized by a Finsler function $F_A:\tilde{N}\to ]-\infty,0]$ of the form
\begin{align}
F_A(x,y)=\,\left(-g(y,y)\right)^{1/2}\,+A\cdot \,y,\quad g(y,y)\leq 0,
\label{Randers space Asanov}
\end{align}
where $g$ is a metric of Lorentzian signature.
As a consequence of Proposition \ref{Randers spinor relation}, we have that
for $y\in\,T_xM$ such that $g(y,y)<0$, the relation
\begin{align}
-\frac{1}{4}\,Tr(\mathcal{M}\cdot \mathcal{F}_A)=\,F_A(x,y)
\label{spinor form of Randers space}
\end{align}
holds good.
Similarly, if $g(y,y)>0$, then the expression
\begin{align}
\frac{1}{4}\,F(\mathcal{M}\cdot\tilde{\mathcal{F}}_A)=\,F_A(x,y)
\label{spinor form of Randers space 2}
\end{align}
holds good.
Also note that
\begin{align*}
-\lim_{g(y,y)\to 0}\frac{1}{4}\,Tr(\mathcal{M}\cdot \mathcal{F}_A)=\,\lim_{g(y,y)\to 0}\frac{1}{4}\,Tr(\mathcal{M}\cdot \tilde{\mathcal{F}}_A)=\,A_j\,y^j.
\end{align*}
Therefore, it is natural to define
\begin{align}
F_A(x,y)=\,A_j\,y^j,\quad \textrm{if}\quad g(y,y)=0.
\label{spinor form of Randers space 3}
\end{align}
Gluing together the relations \eqref{spinor form of Randers space}, \eqref{spinor form of Randers space 2} and \eqref{spinor form of Randers space 3} we have an example of Randers spacetime as introduced in the sheave-theoretical approach to Finsler spacetimes introduced in \cite{Ricardo Randers-Lorentz}. To avoid degeneracies of the Finsler structure associated with Lorentz transformations, it is convenient to consider local $1$-forms $A$ defined on open subsets of $M$.
For each of the local representatives
$A\in\,[\tilde{A}]$ one can define in an open set $U\subset \,M$ the
function ${F}_{A}:TM\setminus\{0\}\to\,\mathbb{R}$ as
\begin{align}
{F}_A(x,y):=\left\{
\begin{array}{l l}
-\frac{1}{4}\,Tr(\mathcal{M}\cdot \mathcal{F}_A) &\quad \textrm{for  } \,\,g(y,y)<\,0, \\
-\lim_{n\to +\infty}\,\frac{1}{4}\,Tr(\mathcal{M}(y_n)\cdot \mathcal{F}_A(y_n)) &  \quad\textrm{for} \,\{y_n\}\to y \textrm{ with}\\
 & \quad g(y,y)=0,\\
\frac{1}{4}\,Tr(\mathcal{M}\cdot\tilde{\mathcal{F}}_A) &\quad \textrm{for  } \,\,g(y,y) > 0. \\
\end{array} \right.
\label{gauge Randers space}
\end{align}
Note that
\begin{align*}
\lim_{n\to +\infty}\,Tr(\mathcal{M}(y_n)\cdot \mathcal{F}_A(y_n))=\,\lim_{n\to +\infty}\,Tr(\mathcal{M}(y_n)\cdot \tilde{\mathcal{F}}_A(y_n))
\end{align*}
for the sequence $\{y_n\}\to y$ with $g(y,y)=0$.

Let us consider now the second order homogeneous formulation of Randers spaces as discussed in \cite{Hohmann Pfeifer Voicu}. The Lagrangian function is of the form
\begin{align}
\widetilde{L}_A=\,\left(|g(y,y)|^{1/2}+\,A\cdot y\right)^2 .
\label{second order label Randers}
\end{align}
Then it is easy to show that
\begin{align}
\widetilde{L}_A(x,y):=\left\{
\begin{array}{l l}
\frac{1}{4^2}\,Tr^2(\mathcal{M}\cdot \tilde{\mathcal{F}}_A) &\quad \textrm{for  } \,\,g(y,y)<\,0, \\
\lim_{n\to +\infty}\,\frac{1}{4^2}\,Tr^2(\mathcal{M}(y_n)\cdot \mathcal{F}_A(y_n)) &  \quad\textrm{for} \,\{y_n\}\to y \textrm{ with}\\
 & \quad g(y,y)=0,\\
\frac{1}{4^2}\,Tr^2(\mathcal{M}\cdot\mathcal{F}_A) &\quad \textrm{for  } \,\,g(y,y) > 0. \\
\end{array} \right.
\label{second order Randers in terms of Clifford algebra}
\end{align}
The reason of the difference between \eqref{gauge Randers space} and \eqref{second order Randers in terms of Clifford algebra} is the relative minus sign between the two terms $|g(y,y)|^{1/2}$ and  $\,A\cdot y$ in the definitions of $F_A$ and $\widetilde{L}_A$. Also, note that because of the properties of the trace of Dirac matrices,
\begin{align*}
& Tr^2(\mathcal{M}\cdot {\mathcal{F}}_A)=\,4\,Tr(\mathcal{M}\,\cdot \mathcal{M}\cdot {\mathcal{F}}_A\cdot {\mathcal{F}}_A),\\
& Tr^2(\mathcal{M}\cdot \tilde{\mathcal{F}}_A)=\,4\,Tr(\mathcal{M}\cdot \mathcal{M}\cdot\tilde{\mathcal{F}}_A\cdot \tilde{\mathcal{F}}_A).
\end{align*}
Thus $\widetilde{L}$ can be re-written directly in terms of trace of elements of the algebra $\mathcal{A}_\mathbb{C}$,
\begin{proposicion} For the Randers type function \eqref{second order Randers in terms of Clifford algebra} one has that
\begin{align}
\widetilde{L}_A(x,y):=\left\{
\begin{array}{l l}
\frac{1}{4}\, Tr(\mathcal{M}\cdot \mathcal{M}\cdot\tilde{\mathcal{F}}_A\cdot \tilde{\mathcal{F}}_A)(x,y)&\quad \textrm{for  } \,\,g(y,y)<\,0, \\
\lim_{n\to +\infty}\,\frac{1}{4}Tr(\mathcal{M}\cdot \mathcal{M}\cdot\tilde{\mathcal{F}}_A\cdot \tilde{\mathcal{F}}_A)(x,y\,(x,y) &  \quad\textrm{for} \,\{y_n\}\to y \textrm{ with}\\
 & \quad g(y,y)=0,\\
\frac{1}{4}\,Tr(\mathcal{M}\cdot \mathcal{M}\cdot\tilde{\mathcal{F}}_A\cdot \tilde{\mathcal{F}}_A)(x,y) &\quad \textrm{for  } \,\,g(y,y) > 0. \\
\end{array} \right.
\label{second order Randers in terms of Clifford algebra 2}
\end{align}
\end{proposicion}

Let us remark that the relations \eqref{SpinorBeam fluid}, \eqref{SpinorBeam fluid 2}, \eqref{spinor form of Randers space},\eqref{spinor form of Randers space 2}  describe only formal Finsler spacetimes. It is necessary to adopt further restrictions on the $1$-form $A$, in order to satisfy regularity conditions of the associated metric.

The theory sketched above can also be applied to Euclidean signature spaces. If $g$ is a Riemannian metric, instead than a Lorentzian structure, then the same algebra contains elements related with positive definite Finsler spacetimes  of the form $Tr(\mathcal{M}\cdot\tilde{\mathcal{F}}_A)$ and $Tr(\mathcal{F}_A\cdot\,{\mathcal{F}}_A)$.
There are also elements of $\mathcal{A}_{\mathbb{C}}$ associated with relevant norm functions of Lorentzian and Riemannian metrics, namely, $ Tr(\mathcal{M}\cdot\tilde{\mathcal{M}}-\mathbb{I}_4)$ and $Tr(\mathcal{M}\cdot{\mathcal{M}}-\mathbb{I}_4)$. Furthermore, the methodology discussed above suggests that any metric spacetime structure determined by a Lorentzian metric $(M,g)$ and a finite set of $1$-forms on $M$ can be expressed in terms of the trace of elements of  $\mathcal{A}_{\mathbb{C}}$.
\subsection{Trace operator and Finsler type structures}
 The elements that we have used in the extension are the (pseudo)-Riemannian metric $g$, differential forms defined on $M$ and the trace operator $Tr:\mathcal{A}_{\mathbb{C}}\to\mathbb{C}$. Recall that the trace operator has in general the following properties:
 \begin{itemize}
 \item  $Tr\left( M_1\,M_2\right) =\,tr\left( M_2 \,M_1\right)$, for $M_1,\,M_2\in\mathcal{A}_{\mathbb{C}}$,

 \item It is linear, $Tr\left( M_1+\,\lambda\,M_2\right)=\,Tr\left( M_1\right)+\,\lambda \,Tr\left( M_2\right)$.

 \item  It is invariant under unitary transformations,
 \begin{align*}
 Tr\left(U\,M_1\,U^{-1}\right)=\,Tr \left( M_1\right)
 \end{align*}

 \end{itemize}
 These are just the properties required for the construction of our map from elements of $\mathcal{A}_{\mathbb{C}}$ to $\mathbb{C}$. Therefore, the trace operator provides the natural map from $\mathcal{A}_{\mathbb{C }}$ to $\mathcal{A}_{\mathbb{C}}$ to the Finsler type space structures.

\section{First order differential operators and Finsler norms}
The above construction determines several interesting Finsler type spacetime structures. Moreover, the sheaf $\mathcal{A}_{\mathbb{C}}$ is related to geometric first order operators. In order to simplify the treatment from a geometric point of view, let us consider the case when $M$ is four dimensional and the metric $g$ is flat. If the velocity vector $y$ is expressed in terms of canonical momentum as $y=\,p^*/m$, where $p$ is the $4$-momenta of a single particle system and $p^*$ is the dual vector to $p$ via $g$, then one finds  easily a relation between the operator $\mathcal{M}=\,\frac{\gamma_i\,p^i}{m}-\,\mathbb{I }_4$, where ${p}^i=\,m\,y^i$ and the Dirac type operator $\widehat{D}=\,\gamma_i\,\hat{p}^i-m\,\mathbb{I}_4$:
  \begin{align*}
  \mathcal{M}=\,\frac{\gamma_i\,p^i}{m}-\,\mathbb{I }_4\quad \Leftrightarrow\quad \widehat{D}=\,\imath\,\gamma^i\,\partial_i-m\,\mathbb{I}_4,
  \end{align*}
  where we have assumed that $m=\,|g(p,p)|^{1/2}$ and $\hat{p}^i=\,-\imath\,\eta^{ij}\partial_j$. Therefore, the element $\mathcal{M}$ of the algebra $\mathcal{A}_{\mathbb{C}}$ is related with  the Dirac type operator associated to a Dirac particle of mass $m$.

  Similarly, the element $\mathcal{F}_A=\,\left(\gamma_i-\,m\,A_i\,\mathbb{I}_4\right){p}^i\,\in\mathcal{A}_{\mathbb{C}}$ has associated the differential operator
   \begin{align*}
   \widehat{D}_A=\,\imath \left(\gamma^i\,-\,A^i\,\mathbb{I}_4\right)\,\partial_i .
   \end{align*}

Conversely, one can consider differential operators and assign a Finsler type structures. For instance, we can consider the $U(1)$ covariant derivative acting on fermions,
  \begin{align*}
   \widehat{\Gamma}_A:=\,\gamma^i\, \left(\imath\,\partial_i\,-\,A_i\,\mathbb{I}_4\right),
   \end{align*}
where here $A_i$ are the components of the potential $1$-form. The corresponding element of the algebra $\mathcal{A}_{\mathbb{C}}$
\begin{align*}
{\Gamma}_{A,m}:=\,\gamma_i\,y^i -m\,\gamma_i\,A^i\,\mathbb{I}_4.
\end{align*}
Applying the trace operator one can see that the corresponding {\it norm function} is
\begin{align*}
Tr(\mathcal{M}\cdot \Gamma_{A,m})=\,|g(y,y)|^{1/2}-\,\frac{y_i\,A^j}{|g(y,y)|^{1/2}}.
\end{align*}
This is  not an $y$-homogeneous  norm function. The non-homogeneity due to the fact that $\Gamma_A$ is not $y$-homogeneous. This example shows that the theory can be extended to non-homogeneous Finsler type structures in order to catch significant geometric objects.
\section{Conclusion}
In this work we have observed that the fundamental functions $F$ of certain metric Randers type structures can be re-written in terms of a trace operator acting on an extension of Clifford bundle. The construction relies on the existence of an underlying Lorenzian structure $(M.g)$ and $1$-forms defined on $M$. 
Furthermore, in the case when $g$ is flat, we have shown a relation between first order operators of Dirac type and non-homogeneous Finsler metrics.  

There are two limitations in our theory that we should mention. The first is that the theory does not applies to general Finsler spacestimes, since it requires the underlying Lorentzian structure and $1$-forms. Currently, this appear as an essential limitation to the scope of the theory. Second, the relation between first order linear operators and non-homogeneous Finsler spacetimes is limited to the case when $g$ is flat. However, this constraint does not appear as a fundamental one. One of the research lines to be developed consists to have a general formulation of the result for generic Lorentzian underlying metric structure.

The relations discussed in this paper can be developed further in several further directions. From a mathematical point of view, the study of bundles $\mathcal{A}_{\mathbb{C}}$ could reveal interesting geometric facts. From a more physical point of view, the interpretation of Randers type spaces in terms of elements of  $\mathcal{A}_{\mathbb{C}}$ opens the possibility to extend spinorial methods to field theories based on Randers spacetimes. This is of particular interest to explore generalizations of Einstein gravity coupled to matter.

\end{document}